\documentclass[leqno,12pt]{amsart}
\usepackage{latexsym}
\usepackage{mathrsfs}
\usepackage{amsmath}
\usepackage{amssymb}
\usepackage[bookmarks]{hyperref}
\usepackage{pstricks,pst-plot}

\font\tenscr=rsfs10 
\font\sevenscr=rsfs7 
\font\fivescr=rsfs5 
\skewchar\tenscr='177 \skewchar\sevenscr='177 \skewchar\fivescr='177
\newfam\scrfam \textfont\scrfam=\tenscr \scriptfont\scrfam=\sevenscr
\scriptscriptfont\scrfam=\fivescr
\def\scr{\fam\scrfam}

\newtheorem{theorem}{Theorem}[section]
\newtheorem{lemma}[theorem]{Lemma}
\newtheorem{corollary}[theorem]{Corollary}
\newtheorem{proposition}[theorem]{Proposition}

\theoremstyle{definition}

\newtheorem{definition}[theorem]{Definition}

\allowdisplaybreaks[1]

\newcommand{\C}{\mathbb{C}}
\newcommand{\D}{\mathbb{D}}

\newcommand{\Z}{\mathbb{Z}}

\newcommand{\row}[2]{#1_1,\ldots,#1_#2}

\def\Z{\mathbb Z}

\def\C{\mathbb C\/}

\def\sC{{\scr C}}
\def\sF{{\scr F}}
\def\sG{{\scr G}}

\def \ma {\mathfrak{M}_A}

\def\tA{\tilde A}
\def\tX{\tilde X}
\def\tf{\tilde f}
\def\tpi{\tilde \pi}

\def\row#1#2{#1_1,\ldots,#1_#2}

\newcommand{\what}{\widehat}

\newcommand{\ti}{\tilde}

\def\tX{\tilde X}
\def\wX{\widehat X}

\def\wY{\widehat Y}

\def\cma{{\frak M}_A}

\newcommand{\bthm}{\begin{theorem}}
\newcommand{\ethm}{\end{theorem}}
\newcommand{\blem}{\begin{lemma}}
\newcommand{\elem}{\end{lemma}}
\newcommand{\bcor}{\begin{corollary}}
\newcommand{\ecor}{\end{corollary}}
\newcommand{\bprop}{\begin{proposition}}
\newcommand{\eprop}{\end{proposition}}
\newcommand{\bdefn}{\begin{definition}}
\newcommand{\edefn}{\end{definition}}
\newcommand{\bpf}{\begin{proof}}
\newcommand{\epf}{\end{proof}}

\newcommand{\bsoln}{\begin{soln}}
\newcommand{\esoln}{\end{soln}}
\newcommand{\bexc}{\begin{exc}}
\newcommand{\eexc}{\end{exc}}
\newcommand{\bex}{\begin{ex}}
\newcommand{\eex}{\end{ex}}
\newcommand{\brem}{\begin{rem}}
\newcommand{\erem}{\end{rem}}
\newcommand{\bprob}{\begin{prob}}
\newcommand{\eprob}{\end{prob}}
\newcommand{\bclaim}{\begin{claim}}
\newcommand{\eclaim}{\end{claim}}

\newcommand{\bm}{\bibitem}

\newcommand{\mc}{\mathcal}

\newcommand{\mpt}{\mathpunct}

\newcommand{\ra}{\rightarrow}

\newcommand{\ol}{\overline}

\newcommand{\bi}{\begin{itemize}}
\newcommand{\ei}{\end{itemize}}

\newcommand{\bc}{\begin{cases}}
\newcommand{\ec}{\end{cases}}

\newcommand{\ba}{\begin{array}}
\newcommand{\ea}{\end{array}}

\newcommand{\bea}{\begin{eqnarray}}
\newcommand{\eea}{\end{eqnarray}}

\newcommand{\beaa}{\begin{eqnarray*}}
\newcommand{\eeaa}{\end{eqnarray*}}

\newcommand{\beastar}{\begin{eqnarray*}}
\newcommand{\eeastar}{\end{eqnarray*}}

\newcommand{\sste}{\subseteq}

\newcommand{\ptl}{\partial}

\font\tenscr=rsfs10 
\font\sevenscr=rsfs7 
\font\fivescr=rsfs5 
\skewchar\tenscr='177 \skewchar\sevenscr='177 \skewchar\fivescr='177
\newfam\scrfam \textfont\scrfam=\tenscr \scriptfont\scrfam=\sevenscr
\scriptscriptfont\scrfam=\fivescr
\def\scr{\fam\scrfam}

\def\scr{\fam\scrfam}

\begin{document}

\title[A Hull with No Nontrivial Gleason Parts]{A Hull with No Nontrivial Gleason Parts}
\author{Brian J. Cole}
\address{Department of Mathematics, Brown University, Providence, RI 02912}
\email{bjc@math.brown.edu}
\author{Swarup N. Ghosh}
\address{Department of Mathematics, Southwestern Oklahoma State University, Weatherford, OK 73096}
\email{swarup.ghosh@swosu.edu}
\author{Alexander J. Izzo}
\address{Department of Mathematics and Statistics, Bowling Green State University, Bowling Green, OH 43403}
\email{aizzo@math.bgsu.edu}


\begin{abstract}
\vskip 24pt
The existence of a nontrivial polynomially convex hull with every point a one-point Gleason part and with no nonzero bounded point derivations is established.  This strengthens the \hbox{celebrated} result of Stolzenberg that there exists a nontrivial polynomially convex hull that contains no analytic discs.  A doubly generated counterexample to the peak point conjecture is also presented.
\end{abstract}
\maketitle


\section{Introduction}

It was once conjectured that whenever the polynomially convex hull $\wX$ of a compact set $X$ in $\C^n$ is strictly larger than $X$, the complementary set $\wX\setminus X$ must contain an analytic disc.  This conjecture was disproved by Gabriel Stolzenberg \cite{Stol}.  Given Stolzenberg's result, it is natural to ask whether weaker semblances of analyticity, such as nontivial Gleason parts or nonzero bounded point derivations, must be present in $\wX\setminus X$.  We will establish the following result giving a negative answer to this question.

\begin{theorem}\label{maintheorem}
There exists a compact set $X$ in $\C^3$ such that $\wX\setminus X$ is nonempty but every point of $\wX$ is a one-point Gleason part for $P(X)$ and there are no nonzero bounded point derivations on $P(X)$.
\end{theorem}

It is particularly surprising that the question of whether nontrivial polynomially convex hulls must contain nontrivial Gleason parts has not been considered in the literature earlier because early efforts to prove the existence of analytic discs focused on the use of Gleason parts.  John Wermer \cite{Wermer} proved that every nontrivial Gleason part for a Dirichlet algebra is an analytic disc, and this was extended to uniform algebras with uniqueness of representing measures by Gunter Lumer \cite{Lumer}.  However, once it was proven that analytic discs  do not always exist, it seems that almost no further work was done on Gleason parts in hulls except for a result of Richard Basner \cite{Basener} that there exists a compact set in $\C^2$ whose rationally convex hull contains no analytic discs but contains a Gleason part of positive 
4-dimensional measure.  Additional examples of hulls that contain no analytic discs but do contain weaker semblances of analyticity will be given in the paper \cite{Izzofuture} of the third author.

Our proof of Theorem~\ref{maintheorem} uses a construction from the first author's dissertation \cite{Cole} related to the so called peak point conjecture.  This conjecture asserted that if $A$ is a uniform algebra on its maximal ideal space $X$, and if each point of $X$ is a peak point for $A$, then $A=C(X)$, the algebra of all continuous complex-valued functions on $X$.  In his dissertation \cite{Cole}, the first author gave a general construction for, roughly speaking, successively adjoining square roots of functions in a uniform algebra, and he used this construction to give a counterexample to the peak point conjecture.  The construction also yields the existence of uniform algebras $A$ such that the maximal ideal space $\cma$ of $A$ is strictly larger than the space on which $A$ is defined and yet $A$ has no nontrivial Gleason parts and no nonzero point derivations.
Since the algebras produced by this general construction are not in general finitely generated though, this does not address the question of the existence of analytic discs and point derivations in polynomially convex hulls settled by Theorem~\ref{maintheorem} above.  

Following Garth Dales and Joel Feinstein \cite{H.Dales_J.Feinstein_2008}, we will say that a uniform algebra $A$ has dense invertibles if the invertible elements of $A$ are dense in $A$.
Carrying out the general construction in the first author's dissertation \cite{Cole} starting with a suitable doubly generated uniform algebra with dense invertibles and successively adjoining square roots to a countable dense collection of invertible elements leads to a doubly generated counterexample to the peak point conjecture.

\begin{theorem}\label{double}
There exists a compact polynomially convex set $X$ in $\C^2$ such that $P(X)\neq C(X)$ but every point of $X$ is a peak point for~$P(X)$.
\end{theorem}

To apply a similar argument to prove Theorem~\ref{maintheorem} we need a doubly generated uniform algebra with dense invertibles having the additional property that the algebra is defined on a proper subset of its maximal ideal space.  Such an algebra  was constructed by Dales and Feinstein~\cite{H.Dales_J.Feinstein_2008}.

\bthm[\cite{H.Dales_J.Feinstein_2008}, Theorem~2.1]
\label{Dales-Feinstein Algebra}
There exists a compact set $Y \sste \ptl \ol{\D}^2$ such that $(0, 0) \in \widehat Y$, and yet $P(Y )$ has dense invertibles.  (Here $\D$ denotes the open unit disc in the plane.)
\ethm

Because the property of having dense invertibles is preserved under the process of adjoining square roots \cite[Theorem~2.1 and Proposition~2.5]{T.Dawson_J.Feinstein_2003}, we get as a corollary of the proofs of Theorems~\ref{maintheorem} and~\ref{double} the following additional assertion.

\bthm \label{dense}
In each of Theorems~\ref{maintheorem} and~\ref{double}, the set $X$ can be chosen so that, in addition, $P(X)$ has dense invertibles.
\ethm

In connection with the proof of their result quoted above,
Dales and Feinstein asked what can be said about the existence of nonzero point derivations \cite[Section~4, Question~1]{H.Dales_J.Feinstein_2008}.  
Although they asked specifically about the algebra $P(Y)$ for the set $Y$ that they constructed, the spirit of their question is what can be said about point derivations on uniform algebras of the form $P(X)$ ($X\subset\C^n$ compact) at points of the set $\wX\setminus X$, when $X$ is such that $\wX\setminus X$ is nonempty while $P(X)$ has dense invertibles.  Taken together, Theorem~\ref{dense} above and results in the third author's paper \cite{Izzofuture} answer this question in the case of bounded point derivations: there exist such uniform algebras both with, and without, bounded point derivations.  Theorem~\ref{dense} and results in \cite{Izzofuture} also show that there exist such uniform algebras both with, and without, nontrivial Gleason parts.
If one considers uniform algebras on nonmetrizable spaces, then one can obtain a uniform algebra with no nonzero, possibly unbounded, point derivations and having dense invertibles.  This can be shown by starting with $P(Y)$, where $Y$ is the set of Dales and Feinstein in Theorem~\ref{Dales-Feinstein Algebra}, and applying the construction of Cole given in \cite{Cole} that produces uniform algebras in which every element has a square root.  The proof is essentially the same as that of \cite[Theorem~2.3]{H.Dales_J.Feinstein_2008}.  It seems to be a difficult open question whether there exist nontrivial uniform algebras (with or without dense invertibles) on a \emph{metrizable} space having no nonzero, possibly unbounded, point derivations.

In the next section we recall some definitions and notations already used above, and we prove a key lemma.  Theorem~\ref{maintheorem} is proved in Section~3, and Theorem~\ref{double} is proved in Section~4.
The arguments given in those sections also contain the proof of Theorem~\ref{dense}.


\section{Preliminaries}~\label{prelim}

For
$X$ a compact (Hausdorff) space, we denote by $C(X)$ the algebra of all continuous complex-valued functions on $X$ with the supremum norm
$ \|f\|_{X} = \sup\{ |f(x)| : x \in X \}$.  A \emph{uniform algebra} on $X$ is a closed subalgebra of $C(X)$ that contains the constant functions and separates
the points of $X$.  

For a compact set $X$ in $\C^n$, we denote by 
$P(X)$ the uniform closure on $X$ of the polynomials in the complex coordinate functions $z_1,\ldots, z_n$, and we denote by $R(X)$ the uniform closure of the rational functions with no poles on $X$.  It is well known that the maximal ideal space of $P(X)$ can be naturally identified with the \emph{polynomially convex hull} $\what X$ of $X$ defined by
$$\what X=\{z\in\C^n:|p(z)|\leq \max_{x\in X}|p(x)|\
\mbox{\rm{for\, all\, polynomials}}\, p
\}.$$

Let $A$ be a uniform algebra on a compact space $X$.
A point $x\in X$ is said to be a \emph{peak point} for $A$ if
there exists $f \in A$ with $f(x) = 1$ and $|f(y)| < 1$ for all $y \in X \setminus \{x\}$.
The \emph{Gleason parts} for the uniform algebra $A$ are the equivalence classes in the maximal ideal space of $A$ under the equivalence relation $\varphi\sim\psi$ if $\|\varphi-\psi\|<2$ in the norm on the dual space $A^*$.  (That this really is an equivalence relation is well known but {\it not\/} obvious!)
For $\phi$ a multiplicative linear functional on $A$, a \emph{point derivation} on $A$ at $\phi$ is a linear functional $\psi$ on $A$ satisfying the identity 
$$\phantom{\hbox{for all\ } f,g\in A.} \psi(fg)=\psi(f)\phi(g) + \phi(f)\psi(g)\qquad \hbox{for all\ } f,g\in A.$$
A point derivation is said to be \emph{bounded} if it is bounded (continuous) as a linear functional.
It is well known that every peak point is a one-point Gleason part, and that at a peak point there are no nonzero point derivations.  (See for instance \cite{Browder}.)

An \emph{analytic disc} in the maximal ideal space $\ma$ of a uniform algebra $A$ is, by definition, an injective map $\sigma$ of the open unit disc $\D\subset \C$  into 
$\ma$ such that the function $f\circ \sigma$ is analytic on $\D$ for every $f$ in 
$A$.  It is immediate that the presence of an analytic disc implies the existence of a nontrivial Gleason part and nonzero bounded point derivations.

Given continuous functions $\row fk$ on a compact space $X$, we denote by $[\row fk]$ the uniform algebra that they generate, i.e., the smallest closed, unital subalgebra of $C(X)$ that contains $f_1,\ldots, f_k$.  Equivalently, $[\row fk]$ is the uniform closure on $X$ of the set of polynomials in $f_1,\ldots, f_k$.  We say that the uniform algebra $A$ is \emph{doubly generated} if $A$ has a set of generators consisting of two elements.

The following key lemma is a generalization of the theorem, due to Kenneth Hoffman and Errett Bishop in case $n=1$, and Hugo Rossi~\cite{Rossi} in general, that for $X$ a compact set in $\C^n$, the uniform algebra $R(X)$ is generated by $n+1$ functions.  

\begin{lemma}\label{generators}
Let $A$ be a uniform algebra on a compact space $X$.  Suppose that $A$ is generated by a sequence of functions ${\mc F}=\{f_1, f_2, \ldots\}$.  Suppose also that there is a positive integer $N$ such that for each $n>N$, the function $f_n$ is either the inverse of a function in the uniform algebra $[\row f{{n-1}}]$ or is an $m$-th root {\rm(}for some $m\geq 2${\rm)} of an invertible function in the uniform algebra $[\row f{{n-1}}]$.  Then $A$ is generated by $N+1$ functions.
\end{lemma}

In the proof we will make repeated use of the following very well known elementary fact \cite[Theorem~1.2.1]{Browder}.

\begin{proposition}\label{elem}
Let $A$ be a Banach algebra with identity element $e$.  If $f\in A$, $\lambda\in \C$, and $\|f\|<|\lambda|$, then $\lambda e-f$ is invertible in $A$.
\end{proposition}

We will also invoke the following easy lemma.

\blem
\label{arg}

Let $a$ be a positive real number, let $b$ and $c$ be complex numbers, and let $m\geq 2$ be a positive integer.
If $|\arg(c)| < {\pi}/{2m}$, and $|b| < (a/2) \sin(\pi/2m)$, then  
$\bigl|\arg[(a + b)c]\bigr|<\pi/m$.  {\rm(}Here $\arg(z)$ denotes the argument of $z$ chosen to lie in the interval $[-\pi, \pi]$.{\rm)}

\elem

\bpf
Note that
$$|\sin (\arg(a+b))|=\frac{|{\rm Im}\, b|}{|a+b|}\leq \frac{|b|}{|a|-|b|}\leq \frac{|b|}{a/2}
<\sin(\pi/2m).$$
Therefore, 
$$|\arg(a+b)|<\pi/2m.$$
Hence,
$$|\arg((a+b)c)|<\pi/2m+\pi/2m=\pi/m.$$
\epf

\begin{proof}[Proof of Lemma~\ref{generators}]
We will show that if $(c_n)$ is a sequence of positive numbers decreasing to zero sufficiently rapidly, and if $b_{N+1}=\sum_{n=N+1}^\infty c_nf_n$, then $A$ is generated by the functions $f_1,\ldots, f_N, b_{N+1}$.  For $n>N$, in the case $f_n$ is the inverse of a function in 
 $[ f_1, \ldots, f_{n-1}]$, we set $h_n = f^{-1}_n$.
Otherwise $f_n$ is an $m$-th root of an invertible function $g_n$ in $[ f_1, \ldots, f_{n-1}]$, and we then set $h_n = g^{-1}_n$.
Obviously, in either case, $h_n$ is in $[ f_1, \ldots, f_{n-1}]$.
Now choose a sequence $\{c_n: n\geq N+1\}$ of constants such that:
\bi
\item[(i)] $0 < c_n \le 1$ for $n \geq N+1$,
\smallskip
\item[(ii)] $c_n \|f_n\| < 2^{-n}$ for $n \geq N+1$,
\smallskip
\item[(iii)] 
$c_n \|f_n h_k\| < \bc
2^{-n} c_k & \mbox{ if } f_k= h_k^{-1}\\
2^{-n} c^m_k & \mbox{ if } f^m_k = g_k
\ec \indent\quad N+1 \le k < n$
\smallskip
\item[(iv)] if $f^m_k = g_k$, then we require also that
$$\left\| \sum_{n = k+1}^{\infty} c_n f_n \right\| <  (c_k/2)\bigl(\sin(\pi/2m)\bigr)\min_{x\in X} |f_k(x)|, \indent k \geq N+1.$$
(Note that $\min\limits_{x\in X} |f_k|>0$ since $f_k^{-1}$ is defined on $X$.)

\ei
Now define, for $k \geq N+1$,
$$b_k = \sum_{n = k}^{\infty} c_n f_n.$$
Condition (ii) assures that each $b_k$ is in $A$.
Also, using condition (iii), we see that for $k \geq N+1$,
$$ \|b_{k+1} h_k\| \le \sum_{n = k+1}^{\infty} c_n \|f_n h_k\| \le \bc
2^{-k} c_k & \mbox{ if } f_k= h_k^{-1}\\
2^{-k} c^m_k & \mbox{ if } f^m_k = g_k.
\ec$$
Let $B = [f_1, \ldots, f_N, b_{N+1}]$.
Clearly $B \sste A$.
We will prove that 
$A = B$, by using induction to show that $f_n$ is in $B$ for each $n = 1, 2, 3,  \ldots$.
By definition $f_1, \ldots, f_N$ are in $B$.
Now given $k\geq N+1$, we assume that $f_1, \ldots, f_{k-1}$ are in $B$, and show that then $f_k$ is in $B$.
First note that $b_k = b_{N+1} - (c_{N+1} f_{N+1} + \ldots + c_{k-1} f_{k-1})$ is in $B$.
Also $h_k$ is in $B$ since $h_k$ is in $[ f_1, \ldots, f_{k-1}]$.\\

\noindent
Case I: $f_k=h_k^{-1}$\\
Clearly $b_k h_k$ is in $B$, and
\[
\ba{lcl}
\|c_k - b_k h_k\|
&=& \|c_k - \sum_{n = k}^{\infty} c_n f_n h_k\|\\
&=& \|\sum_{n = k+1}^{\infty} c_n f_n h_k\|
= \|b_{k+1} h_k\| \le 2^{-k} c_k < c_k.
\ea
\]
Therefore, by Proposition~\ref{elem}, $b_k h_k = c_k - (c_k -b_k h_k)$ is invertible in $B$.
Hence $h_k$ is invertible in $B$, that is, $f_k = h^{-1}_k$ is in $B$.\\

\noindent
Case II: $f^m_k = g_k$\\
Clearly $b^m_k h_k$ is in $B$, and
\[
\ba{lcl}
\displaystyle
b^m_k h_k
&=& (c_k f_k + b_{k+1})^m h_k\\
&=&\displaystyle c^m_k + b_{k+1} h_k \left[ \binom{m}{1} (c_k f_k)^{m-1} + \binom{m}{2} (c_k f_k)^{m-2} b_{k+1} + \ldots + b^{m-1}_{k+1} \right].
\ea
\]
Note that by conditions (ii) and (iv) above $\|b_{k+1}\|< 2^{-k}$, so
we can bound the second term in the last line of the preceding display as follows:
\beaa
\begin{aligned}
\Biggl\| b_{k+1} h_k &\left[ \binom{m}{1} (c_k f_k)^{m-1} + \binom{m}{2} (c_k f_k)^{m-2} b_{k+1} + \ldots + b^{m-1}_{k+1} \right] \Biggr\|\\
\le&\bigl \|b_{k+1} h_k\bigr\| \left[ \binom{m}{1} (2^{-k})^{m-1} + \binom{m}{2} (2^{-k})^{m-2} 2^{-k} + \ldots + (2^{-k})^{m-1} \right]\\
\le& 2^{-k} c^m_k (2^{-k})^{m-1} (2^m - 1)\\
=& (2^m - 1) 2^{-mk} c^m_k\\
<& c^m_k.
\end{aligned}
\eeaa
Therefore, applying Proposition~\ref{elem} again gives that $b^m_k h_k$ is invertible in $B$.  Furthermore, the inequality gives that the range of the function
$b^m_k h_k$ lies entirely in the open right half-plane.
Let $\sqrt[m]{z}$ denote the branch of the logarithm defined on the open right half-plane taking values in the sector where the argument lies in the interval $(-\pi/2m, \pi/2m)$, and set $s_k=\sqrt[m]{z}\circ (b^m_k h_k)$.  
Since $\sqrt[m]{z}$ can be approximated uniformly on compact sets by polynomials, $s_k$ lies in $B$, and since $b^m_k h_k$ is invertible in $B$, so is $s_k$.

Note that 
$$f^m_k = g_k = h^{-1}_k = (b_k s^{-1}_k)^m.$$
Taking $m$-th roots we get, that for each $x \in X$, there is an $m$-th root of unity $\alpha(x)$ such that
$$\alpha(x) f_k(x) =  b_k(x) s^{-1}_k(x) = \bigl(c_k f_k(x) + b_{k+1}(x)\bigr)s^{-1}_k(x).\leqno(1)$$
Then
$$\alpha(x) =  
\left( c_k + \frac{b_{k+1}(x)}{f_k(x)} \right) s^{-1}_k(x), \indent x \in X.\leqno(2)$$

By construction, the argument of $s_k(x)$ lies in the interval $(-\pi/2m, \pi/2m)$, and hence the same is true of the argument of $s^{-1}_k(x)$.
Also, from condition (iv), we see that $\left| \frac{b_{k+1}(x)}{f_k(x)} \right| < (c_k/2)\sin(\pi/2m)$, for $x \in X$.
Thus by Lemma~\ref{arg} and equation (2), the $m$-th root of unity $\alpha(x)$ satisfies $\bigl|\arg\alpha(x)\bigr|<\pi/m$ and hence must be equal to 1.  Consequently by equation (1), we have $f_k = b_k s^{-1}_k$.
Therefore, $f_k$ is in $B$.\\

Thus, by induction, $f_n$ is in $B$ for each positive integer $n$, and hence, $A = B = [f_1, \ldots, f_N, b_{N+1}]$.
\end{proof}
\eject


\section{A Hull with No Nontrivial Gleason Parts\\ and No Nonzero Bounded Point Derivations}

This section is devoted to the proof of Theorem~\ref{maintheorem} 
and the related part of Theorem~\ref{dense}.

Let $Y$ be the set given in Theorem~\ref{Dales-Feinstein Algebra}.  Set 
$X_0=\wY$ and $A_0=P(X_0)$.  Following the iterative procedure used by the first author \cite[Theorem~2.5]{Cole} (and also described in \cite[p.~ 201]{Stout}) to obtain uniform algebras with only trivial Gleason parts on metrizable spaces, we will define a sequence of uniform algebras $\{A_m\}_{m=0}^\infty$.  First let $\sF_0=\{f_{0,n}\}_{n=1}^\infty$ be a countable dense set of invertible functions in $A_0$.  Let $\pi:X_0\times \C^\omega\rightarrow X_0$ and $\pi_n:X_0\times \C^\omega \rightarrow \C$ denote the projections given by $\pi(x,(y_k)_{k=1}^\infty)=x$ and $\pi_n(x,(y_k)_{k=1}^\infty)=y_n$.  Define $X_1\subset X_0\times \C^\omega$ by
$$X_1=\{z\in X_0\times \C^\omega: \pi_n^2(z)= (f_{0,n}\circ\pi)(z) \hbox{\ for all } n\in \Z_+\},$$
and let $A_1$ be the uniform algebra on $X_1$ generated by 
the functions $\{\pi_n\}_{n=1}^\infty$.
Note that since $\pi_n^2=f_{0,n}\circ\pi$ on $X_1$, the functions $f_{0,n}\circ\pi$ belong to $A_1$.  Thus by identifying each function $f\in A_0$ with $f\circ\pi$, we can regard $A_0$ as a subalgebra of $A_1$.  Then $\pi_n^2=f_{0,n}$, so we will denote $\pi_n$ by $\sqrt{f_{0,n}}$.  By 
\cite[Theorem~2.1]{T.Dawson_J.Feinstein_2003} the algebra $A_1$ has dense invertibles.  The collection $\sC$ of polynomials in the members of 
$\{\sqrt{f_{0,n}}\}_{n=1}^\infty$ with coefficients whose real and imaginary parts are rational is a countable dense subset of $A_1$.  Moreover, because the group of invertibles in $A_1$ is a dense, {\it open} subset of $A_1$, the intersection 
of $\sC$ with the group of invertibles is also dense in $A_1$.  Let $\sF_1=\{f_{1,1}, f_{1,2}, \ldots\}$ be this intersection.  Note that each $\sqrt{f_{0,n}}$ lies in $\sF_1$.  We now iterate this construction to obtain a sequence of uniform algebras $\{A_m\}_{m=0}^\infty$ on compact metric spaces $\{X_m\}_{m=0}^\infty$ and for each $m$ a countable dense set $\sF_m=\{f_{m,1}, f_{m,2}, \ldots\}$ of invertible functions in $A_m$.  Each $A_m$ can be regarded as a subalgebra of $A_{m+1}$.  
Each function in $\sF_{m+1}$ is a polynomial in the members of 
$\{\sqrt{f_{m,n}}\}_{n=1}^\infty$.  In addition, each $\sqrt{f_{m,n}}$ lies in $\sF_{m+1}$.

We now take the direct limit of the system of uniform algebras $\{A_m\}$ to obtain a uniform algebra $A_\omega$ on some compact metric space $X_\omega$.  If we regard each $A_n$ as a subset of $A_\omega$ in the natural way, and set $\sF=\bigcup \sF_m$, then $\sF$ is a dense set of invertibles in $A_\omega$, and every member of $\sF$ has a square root in $\sF$.  It follows (by \cite[Lemma~1.1]{Cole}) that every Gleason part for $A_\omega$ consists of a single point and that there are no nonzero bounded point derivations on $A_\omega$.

By \cite[Theorem~2.5]{Cole}, there is a surjective map $\tpi:X_\omega\rightarrow X_0$ that sends the Shilov boundary for $A_\omega$ into the Shilov boundary for $A_0$.  Since the Shilov boundary for $A_0$ is contained in $Y\subsetneq \wY$, this gives that the Shilov boundary $\Gamma_{A_\omega}$ for $A_\omega$ is a proper subset of $X_\omega$.  Now to complete the proof of Theorem~\ref{maintheorem} and the related part of Theorem~\ref{dense} it suffices to show that $A_\omega$ is generated by three functions, for if $f_1, f_2, f_3$ are generators for $A_\omega$, and we set $X=\{ (f_1(x),f_2(x), f_3(x)): x\in \Gamma_{A_\omega}\}$, then $P(X)$ is isomorphic as a uniform algebra to $A_\omega$ and
$\wX=\{ (f_1(x),f_2(x), f_3(x)): x\in X_\omega\}$
is strictly larger than $X$.

To show that $A_\omega$ is generated by three functions, we consider the following array of functions:
\[
\ba{cccc}
z_1 \circ \ti \pi\\[7pt]
z_2 \circ \ti \pi\\[7pt]
\sqrt{f_{0,1}} & \sqrt{f_{0,2}} & \sqrt{f_{0,3}} & \ldots\\[7pt]
f^{-1}_{1,1} & f^{-1}_{1,2} & f^{-1}_{1,3} & \ldots\\[7pt]
\sqrt{f_{1,1}} & \sqrt{f_{1,2}} & \sqrt{f_{1,3}} & \ldots\\[7pt]
f^{-1}_{2,1} & f^{-1}_{2,2} & f^{-1}_{2,3} & \ldots\\[7pt]
\sqrt{f_{2,1}} & \sqrt{f_{2,2}} & \sqrt{f_{2,3}} & \ldots\\[7pt]
\vdots & \vdots & \vdots & \ddots\\
\ea
\]
The functions in the above array generate $A_\omega$.  Thus by Lemma~\ref{generators}, the proof will be complete if we show that the functions in the array can be arranged in a sequence that begins with
$z_1 \circ \ti \pi$ and
$z_2 \circ \ti \pi$ and is such that each function that comes after that in the sequence is either the inverse of a function in the uniform algebra generated by the preceding functions in the sequence or else is the 
square root of an invertible function in the uniform algebra generated by the preceding functions in the sequence.  Roughly we would like to list the functions by proceeding along successively lower upward slanting diagonals listing
$$\sqrt{f_{0,1}}; f_{1,1}^{-1}, \sqrt{f_{0,2}}; \sqrt{f_{1,1}}, f_{1,2}^{-1}, \sqrt{f_{0,3}};\ldots,$$ but this procedure must be modified to insure that whenever the inverse of a function $f_{m,n}$ is listed, the function $f_{m,n}$ belongs to the uniform algebra generated by the previous functions.  We therefore proceed as follows.  Having listed  $z_1\circ \tpi$ and $z_2\circ \tpi$, we list $\sqrt{f_{0,1}}$ and continue listing functions from the same row ($\sqrt{f_{0,2}}$, $\sqrt{f_{0,3}}$,\,\ldots) until the function $f_{1,1}$ belongs to the uniform algebra generated by the listed functions.  We then list 
$f_{1,1}^{-1}$.  If 
$\sqrt{f_{0,2}}$ has not yet been listed, we then add $\sqrt{f_{0,2}}$ to the list.  We then list $\sqrt{f_{1,1}}$.  Next proceeding along the upward slanting diagonal 
through $\sqrt{f_{1,1}}$, we consider $f_{1,2}^{-1}$.  Before listing $f_{1,2}^{-1}$,
we list more functions of the form $\sqrt{f_{0,n}}$, if needed, until $f_{1,2}$ belongs to the generated uniform algebra.  Then we list $f_{1,2}^{-1}$.  Continuing along this same upward slanting diagonal, we list $\sqrt{f_{0,3}}$ if it has not yet been listed.   We then proceed to the next upward slanting diagonal and consider $f_{2,1}^{-1}$.  Before listing $f_{2,1}^{-1}$, we must insure that $f_{2,1}$ is in the generated uniform algebra.  This may involve listing more functions of the form $\sqrt{f_{m,n}}$ from the row above.  Before listing these, we may have to list more functions from the rows above to insure that $f_{m,n}$ is in the generated uniform algebra.  In general, each time we are to list a function of the form $f_{m,n}^{-1}$ or $\sqrt{f_{m,n}}$, we list functions from earlier rows as needed until the function $f_{m,n}$ belongs to the generated uniform algebra.  In this way, we obtain the desired sequence. \hfill$\square$


\section{A Doubly Generated Counterexample to the\\ Peak Point Conjecture}

This section is devoted to the proof of Theorem~\ref{double} 
and the related part of Theorem~\ref{dense}.  The proof can be carried out by an argument similar to the one used above in Section~3 by replacing the starting algebra $P(X_0)$ used there by the algebra $R(X_0)$ for $X_0$ a compact set in the plane such that $R(X_0)\neq C(X_0)$ and the only Jensen measures for $R(X_0)$ are point masses.  However, we will instead give a modification of this argument based on the presentation of Cole's original counterexample to the peak point conjecture given in the text \cite{Browder} by Andrew Browder.  The argument in \cite{Browder}, the essential features of which are the same as in Cole's original construction, yields the desired algebra in one step eliminating the need for induction and the use of an inverse limit.

We now begin the proof.  We will show that there exists a uniform algebra $\tA$ on a compact space $\tX$ such that
\bi
\item[(i)] $\tA\neq C(\tX)$,
\item[(ii)] the maximal ideal space of $\tA$ is $\tX$,
\item[(iii)] every point of $\tX$ is a peak point for $\tA$,
\item[(iv)] $\tA$ has a dense set of invertibles,  and
\item[(v)] $\tA$ is generated by two functions.
\ei
Theorem~\ref{double} and the related part of Theorem~\ref{dense} then follow at once by choosing generators $f_1$ and $f_2$ for $\tilde A$ and setting $X=\{ (f_1(x),f_2(x)): x\in \tX\}$.

Let $X_0$ be a compact set in the plane such that $R(X_0)\neq C(X_0)$ and the only Jensen measures for $R(X_0)$ are point masses.  (The first example of a set $X_0$ with the prescribed property was given by McKissick \cite{Mc}.  McKissick's example is presented in  \cite[pp.~344--355]{Stout}, and a substantial simplification of part of the argument is given in \cite{Ko}.  Probably the simplest example is the one given in \cite[pp.~193--195]{Browder}.)  Set $A=R(X_0)$.  Since $X_0$ has empty interior, 
it is easy to show that the rational functions with poles off $X_0$ and no zeros on $X_0$ are dense in $A$.
Hence, there is a countable dense subset $\sF = \{f_n: n \in \Z_+\}$ of $A$ consisting of invertible elements where each $f_n$ is a rational function $f_n=p_n/q_n$ with $p_n$ and $q_n$ polynomials that are zero free on $X_0$.

Let $J = \{(m, n): m, n \in \Z_+, m\geq 2\}$, and let $Z = X_0 \times \C^J$. 
A point $z$ of $Z$ can be written in the form $z = (x, y)$ where $x \in X_0$ and $y = (y_{m, n})_{(m, n) \in J}$ with $y_{m, n} \in \C$ for each $(m, n) \in J$.
Let $\pi \mpt: Z \ra X_0$ be the projection given by $\pi(x, y) = x$, and for each $(m, n) \in J$, let $\pi_{m, n} \mpt: Z \ra \C$ be the projection given by $\pi_{m, n}(x, y) = y_{m, n}$.
Define $\tilde X$ by
$$\tilde X = \{z \in Z: f_n(\pi(z)) = \pi^m_{m, n}(z), (m, n) \in J\}.$$
Note that $\tilde X$ is closed in $Z$ and that 
each set $\pi_{m, n}(\tilde X)$ is bounded in $\C$ since each $f_n$ is bounded.
Therefore, the Tychonoff theorem implies that $\ti X$ is compact.
Since $\ti X$ is a subspace of a countable product of metrizable spaces, $\ti X$ is metrizable.

Let $\ti \pi \mpt: \ti X \ra X_0$ and $g_{m, n} \mpt: \ti X \ra \C$ be the restrictions to $\ti X$ of $\pi$ and $\pi_{m, n}$, respectively.
Let 
$$\sG = \{f \circ \ti \pi: f \in A\} \cup \{g_{m, n}: (m, n) \in J\}.$$
Note that $\sG$ contains the constant functions and separates points on $\ti X$.
Let $\ti A$ be the uniform algebra on $\ti X$ generated by $\sG$.

To show that $\ti A$ is generated by two functions, we consider the following array of functions:
\[
\ba{cccc}
z \circ \ti \pi\\[4.5pt]
(1/{q_1}) \circ \ti \pi\ \ & (1/{q_2}) \circ \ti \pi \ \ & (1/{q_3}) \circ \ti \pi \ \ & \ldots\\[4.5pt]
(1/{p_1}) \circ \ti \pi \ \ & (1/{p_2}) \circ \ti \pi \ \ & (1/{p_3}) \circ \ti \pi \ \ & \ldots\\[4.5pt]
g_{2,1} & g_{2,2} & g_{2,3} & \ldots\\[4.5pt]
g_{3,1} & g_{3,2} & g_{3,3} & \ldots\\[4.5pt]
g_{4,1} & g_{4,2} & g_{4,3} & \ldots\\[4.5pt]
\vdots & \vdots & \vdots & \ddots\\
\ea
\]
The functions in the above array generate the uniform algebra $\tA$.  Now arrange the above functions in a sequence
$$z \circ \ti \pi, (1/{q_1}) \circ \ti \pi, (1/{p_1}) \circ \ti \pi, (1/{q_2}) \circ \ti \pi, g_{2,1}, 
(1/{p_2}) \circ \ti \pi, (1/{q_3}) \circ \ti \pi, \ldots$$
where after listing $z \circ \ti \pi$ and  $(1/{q_1}) \circ \ti \pi$, we proceed along successively lower upward slanting diagonals.
Then each function in the sequence beyond the first term is either the inverse of a function in the uniform algebra generated by the preceding functions in the sequence or else is the 
$m$-{th} root ($m \ge 2$) of an invertible function in the uniform algebra generated by the preceding functions in the sequence.
Therefore, by Lemma \ref{generators}, $A$ is generated by two functions.

By \cite[Theorem~2.1]{T.Dawson_J.Feinstein_2003}, for each finite subset $J_0$ of $J$, the algebra generated by 
$\sG = \{f \circ \ti \pi: f \in A\} \cup \{g_{m, n}: (m, n) \in J_0\}$
has dense invertibles, and hence so does $\tA$.

For the proof that $\tA$ has properties (i), (ii), and (iii) we follow closely the exposition of the first author's original counterexample to the peak point conjecture given in Browder's text \cite[Appendix]{Browder}.
To verify (ii), let $\ti \phi$ be an arbitrary multiplicative linear functional on $\ti A$.
Then the map $\phi \mpt: A \ra \C$ defined by $\phi (f) = \ti \phi (f \circ \ti \pi)$ is a multiplicative linear functional on $A$, so there exists a point $x \in X_0$ such that $\phi (f) = f(x)$ for all $f \in A$.
Let $z= \bigl(x, \bigl(\tilde\phi(g_{m,n})\bigr)_{(m,n)\in J}\bigr) \in X_0\times \C^J=Z$.  Then $[\pi_{m, n} (z)]^m = [\ti \phi (g_{m, n})]^m = \ti \phi (g^m_{m, n}) = \ti \phi (f_n \circ \ti \pi) = \phi(f_n) = f_n (\pi (z))$, and hence $z$ is a point of $\ti X$.
Furthermore $\ti \phi (h) = h(z)$ for every $h \in \sG$, and hence for every $h \in \ti A$.
This shows that the maximal ideal space of $\ti A$ is $\ti X$.\\

We next verify (iii).
Let $\ti x =(x,y)\in \ti X\subset X_0\times \C^J$, and $\mu$ be a representing measure for $\ti x$.
Let $\ti \pi_*(\mu)$ be the push forward measure of $\mu$ under $\ti \pi$.
Since $\mu$ is a positive measure, supp$(\mu) \sste \ti \pi^{-1}($supp$(\ti \pi_*(\mu))$.
For any $(m, n) \in J$, we have
\[
\ba{lcl}
|f_m(x)|^{1/m}
&=& |g_{m, n}(\ti x)| = \left |\int g_{m, n }\, d\mu \right |\\
&\le& \int |g_{m, n}| \, d\mu = \int |f_n\circ \ti \pi |^{1/m} \, d\mu = \int |f_n|^{1/m} \, d\ti \pi_*(\mu).
\ea
\]
Thus $|f_n(x)| \le \left [\int |f_n|^{1/m} \, d \ti \pi_*(\mu) \right ]^m$ for all $(m, n) \in J$.
Since $\sF=\{f_n\}$ is dense in $A$, it follows that $|f(x)| \le \left [\int |f|^{1/m} \, d \ti \pi_*(\mu) \right ]^m$ for all $f \in A$ and for all $m \in \Z_+$.
It follows that $|f(x)| \le \exp \int \log|f| \, d \ti \pi_*(\mu)$ for all $f \in A$ (see \cite[pp.~125-126]{Browder}).
Thus $\ti \pi_*(\mu)$ is a Jensen measure for $x$ with respect to $A$, and hence, $\ti \pi_*(\mu)$ is the unit point mass at $x$.
It follows that supp$(\mu) \sste \ti \pi^{-1}(\{x\})$.
Next for each $(m, n) \in J$,
$$\left |\int g_{m, n }\, d\mu \right | = |g_{m, n}(\ti x)|
= |f_n(x)|^{1/m} = \int |g_{m, n}| \, d\mu,$$
since $|f_n(x)|^{1/m} = |g_{m, n}|$ on $\ti \pi^{-1}(\{x\})$.
Hence $g_{m, n}$ is constant on supp$(\mu)$ for every $(m, n) \in J$.
Since $\{g_{m, n}: (m, n) \in J\}$ separates the points of $\ti \pi^{-1}(\{x\})$, it follows that supp$(\mu) = \{\ti x\}$.
We conclude that the unit point mass at $\tilde x$ is the only representing measure for $\ti x$ with respect to $\ti A$.
Since $\ti X$ is metrizable, it follows that $\ti x$ is a peak point for $\ti A$.

Finally we prove (i).  For this it suffices to show that if $f\in C(X_0)$ and $f\circ \pi \in \tA$, then $f\in A$.
We will show more: there exists a continuous linear map $P \mpt: C(\ti X) \ra C(X_0)$ that is onto and such that $P(f \circ \ti \pi) = f$ for each $f \in C(X_0)$, and $P(\ti A) = A$.
We define $P$ by ``averaging over the fibers of $\pi$".

For each $(m, n) \in J$, let $H_{m, n}$ be the group of $m$-th roots of unity, 
and let $H = \prod_{(m, n) \in J} H_{m, n}$.
Then $H$ acts on $\tX$ by the action $\gamma \mpt: H \times \tX \ra \tX$ given by $\gamma (h, (x, (y_{m,n}))) = (x, (h_{m,n}\, y_{m,n}))$.
Note that $\gamma$ is a well-defined continuous mapping, and hence is a group action, and write $h\cdot(x, y)$ for $\gamma (h, (x, y))$.
Note also that $H$ is a compact group.
Let $\mu$ be the normalized Haar measure on $H$.
For $\tf \in C(\ti X)$, we define $P\tf \mpt: X_0 \ra \C$ by
$$P\tf(x) = \int_H \tf(h\cdot(x, y))\, d\mu (h).$$
We will show that $P\tf$ is a continuous function of $x$, and is well-defined independent of the choice of $y$, and hence we have a well-defined map $P \mpt: C(\ti X) \ra C(X_0)$.
First note that the integrand is a continuous function of $h$ on $H$, so the integral exists.
Note that $H$ acts transitively on $\ti \pi^{-1} (x)$.
Therefore, if $y'$ is a point in $\C^J$ such that $(x, y') \in \ti X$, then $(x, y') = h'\cdot(x, y)$ for some $h' \in H$.
Then
\beastar
\int_H \tf(h\cdot (x, y'))\, d\mu (h)
&=& \int_H \tf(h\cdot (h'\cdot(x, y)))\, d\mu (h)\\
&=& \int_H \tf(hh'\cdot(x, y))\, d\mu (h)\\
&=& \int_H \tf(h\cdot (x, y))\, d\mu (h),
\eeastar
by the invariance of Haar measure.
Hence $P\tf(x)$ is well-defined independent of the choice of $y$.
To establish the continuity of $P\tf$, let $(x,y)$ be a point of $\tX$ and $(x_n,y_n)$ be a sequence such that 
$(x_n, y_n) \rightarrow (x, y)$ in $\tX$.
For each fixed $h \in H$, note that $\tf(h\cdot (x_n, y_n)) \rightarrow \tf(h\cdot (x, y))$.
Thus the dominated convergence theorem implies that $P\tf(x_n) \rightarrow P\tf(x)$.  Consequently, $P\tf\circ \tilde \pi$ is continuous, and hence so is $P\tf$.

Clearly the map $P$ is linear.
Since $\|\mu\| = 1$, we have $\|P\tf\|_\infty \le \|\tf\|_\infty$, so $P$ is continuous.
Also it is clear that $P(f \circ \ti \pi) = f$ for $f \in C(X_0)$, and hence $P$ is onto.
To show $P(\ti A) \sste A$, it suffices, by the continuity and linearity of $P$, to prove that $P$ sends each monomial in the functions in $\sG$ into $A$.
Furthermore because $g_{m, n}^m=f\circ \tilde\pi$ for each $(m, n) \in J$, it suffices to consider functions of the form $\tf = (f \circ \ti \pi) g_{m_1, n_1}^{r_1}, \ldots, g_{m_k, n_k}^{r_k}$, with $f \in A$ and $r_j<m_j$ for each $j$.
If none of the $g_{m_j, n_j}$ are present, then we have already noted that $P(\tf) = P(f \circ \ti \pi) = f \in A$.
If some $g_{m_j, n_j}$ are present, define $a \in H$ by $a_{m, n} = 1$ for all $(m, n) \ne (m_1, n_1)$ and $a_{m_1, n_1} = e^{2\pi i/m_1}$.
Then setting $\tf_a (x, y) = \tf(a\cdot(x, y))$, we have $\tf_a = a_{m_1, n_1}\tf$, and so $P(\tf_a) = a_{m_1, n_1}P(\tf)$.
But $P(\tf_a) = P(\tf)$ by the invariance of Haar measure.
We conclude that $P(\tf) =0$, so $P(\tf)$ is in~$A$.
\hfill$\square$


\begin{thebibliography}{BC87}

\bibitem{Basener} R. F. Basener, {\it On rationally convex hulls},
Trans.\ Amer.\ Math.\ Soc.\ {\bf 182} (1973), 353--381.

\bm{Browder}
A. Browder, 
{\it Introduction to Function Algebras\/},
Benjamin, New York, 1969.

\bm{Cole}
B. J. Cole, {\it One-point parts and peak point conjecture\/}, Ph.D. dissertation, Yale University, 1968.

\bm{H.Dales_J.Feinstein_2008}
H. G. Dales and J. F. Feinstein, {\it Banach function algebras with dense invertible group\/}, Proc.\ Amer.\ Math.\ Soc., {\bf 136\/} (2008), 1295 - 1304.

\bm{T.Dawson_J.Feinstein_2003}
T. W. Dawson and J. F. Feinstein, {\it On the denseness of the invertible group in Banach algebras\/}, Proc.\ Amer.\ Math.\ Soc.,  {\bf 131\/} (2003), 2831 - 2839. 

\bibitem{Izzofuture} A. J. Izzo, {\it Point derivations and Gleason parts for uniform algebras with dense invertible group\/}, in preparation.

\bibitem{Ko} T. W. K\"orner, {\it A cheaper Swiss cheese\/}, Studia Math. {\bf 83\/} (1986), 33-36.

\bibitem{Mc} R. McKissick, {\it A nontrivial normal sup norm algebra\/}, Bull.\ Amer.\ Math.\ Soc.\ {\bf 69\/} (1963), 391--395.

\bibitem{Lumer} G. Lumer,
{\it Analytic functions and the Dirichlet problem\/}, Bull.\ Amer.\ Math.\ Soc.\ {\bf 70} (1964), 98--104.

\bibitem{Rossi} H. Rossi,
{\it Holomorphically convex sets in several complex variables\/}, Ann.\ of 
Math.\ {\bf 74} (1961), 740--793.


\bibitem{Stol}  G. Stolzenberg, {\it A hull with no analytic structure\/},
J.\ Math.\ Mech.\  {\bf 12}, (1963), 103--111.

\bm{Stout} E. L. Stout, {\it The Theory of Uniform Algebras\/},
Bogden \& Quigley, New York, 1971.

 \bibitem{Wermer} J. Wermer, {\it Dirichlet algebras\/},
 Duke Math.\ J.\ {\bf 27} (1960), 373--382.



\end{thebibliography}
\end{document}